\documentclass{amsart}

\usepackage{amssymb}

\newcommand{\Hbt}{\mathcal{H}}
\newcommand{\M}{\mathcal{M}}
\newcommand{\Mtilde}{\widetilde{\mathcal{M}}}

\newtheorem{theorem}{Theorem}[section]
\newtheorem{lemma}[theorem]{Lemma}
\newtheorem{cor}[theorem]{Corollary}
\newtheorem{prop}[theorem]{Proposition}

\theoremstyle{definition}
\newtheorem{definition}[theorem]{Definition}

\theoremstyle{remark}
\newtheorem{remark}[theorem]{Remark}

\numberwithin{equation}{section}

\begin{document}

\title{An alternative approach to weighted non-commutative Banach function spaces}


\author{Claud Steyn}
\address{DST-NRF CoE in Math. and Stat. Sci,\\ Unit for BMI,\\ Internal Box 209, School of Math. \& Stat. Sci.\\ NWU, PVT. BAG X6001, 2520 Potchefstroom\\ South Africa}
\curraddr{}
\email{29576393@NWU.ac.za}
\thanks{}


\subjclass[2010]{Primary 46L52, 46L51}

\date{}

\dedicatory{}

\begin{abstract} We use a weighted analogue of a trace to define a weighted non-commutative decreasing rearrangement and show it's relationship with the singular value function. We further show an alternative approach to constructing weighted non-commutative Banach function spaces using weighted non-commutative decreasing rearrangements and prove that this approach is equivalent to the original approach by Labuschagne and Majewski.\\
Key Words. weighted, operator algebras, decreasing rearrangement, Banach function spaces, measurable operators
\end{abstract}

\maketitle

\section{Introduction}

The theory of non-commutative Banach function spaces has become a well developed and important area in real analysis. The papers of Fack and Kosaki \cite{FK}, Dodds, Dodds and de Pagter \cite{DDdP1}, and, more recently, Labuschagne and Majewski \cite{LM} are of particular interest for the context of this paper.

Let $\M$ be a semi-finite von Neumann algebra equipped with a semifinite, faithful, normal trace $\tau$. We will denote the $\tau$-measurable operators as $\Mtilde$. We will follow Dodds, Dodds and de Pagter's prescription for constructing non-commutative Banach function spaces, as can be seen in \cite{DDdP1}. We define the non-commutative decreasing rearrangement of $a\in\Mtilde$ as the function\\ $\mu(a):(0,\infty)\mapsto(0,\infty]:t\mapsto \mu_t(a)$ where $\mu_t(a)=\inf\{\Vert ap\Vert:p\in\M_P,\text{ }\tau(1-p)\leq 1\}$. Then, for a Banach function norm $\rho$, the non-commutative Banach function space is given by $L^\rho(\Mtilde)=\{a\in\Mtilde:\mu(a)\in L^\rho(\mathbb{R}^+)\}$ with the norm given by $\Vert a\Vert_\rho = \rho(\mu(a))$. 

In \cite{LM} Labuschagne and Majewski expanded on this by introducing weighted non-commutative Banach function spaces. Suppose $x\in L^1(\Mtilde)$. In the context of \cite{LM} $x$ can be thought of as a state, so then $\tau(x)=1$, though we won't be making this assumption in the present article. For a given Banach function norm, but this time defined on $L^0(\mathbb{R}^+,\mu_t(x)dt)$, the weighted non-commutative Banach function spaces were defined to be the space of all those $a\in\M$ such that $\mu(a)\in L^\rho(\mathbb{R}^+,\mu_t(x)dt)$, with the norm given by $\Vert a\Vert_\rho = \rho(\mu(a))$ as before.

The authors went on to prove that the weighted non-commutative Banach function spaces are in fact Banach spaces that inject continuously into $\Mtilde$ \cite[Theorem 3.7]{LM}. Throughout the proof the map $\tau_x:\Mtilde\mapsto \mathbb{R}:a\mapsto \int_0^\infty\mu_t(a)\mu_t(x)dt$ was used implicitly. Labuschagne and Majewski showed that this map has all the properties of a finite normal faithful trace, with the exception of only being subadditive, i.e. $\tau_x(a+b)\leq\tau(a)+\tau(b)$ for all $a,b\in\Mtilde$.

The map $\tau_x$ will be one of the main foci of this article. Our first main objective will be to define the concept of a $\tau_x$-measurable operator and investigate the relationship between such operators and $\tau$-measurable operators.

Once we have established this relationship, we will define a weighted non-commutative decreasing rearrangement based on the map $\tau_x$. In \cite{LM} the authors assumed $x\in L^1(\Mtilde,\tau)$, but we will relax this assumption to $x\in L^1(\Mtilde,\tau)+\M$.
\\

In \cite{FK} Fack and Kosaki explored some of the properties of non-commutative decreasing rearrangements, in particular a number of convergence theorems and their relationship with $L^p$ norms. These results form some of the fundamentals of the theory of non-commutative decreasing rearrangements and are used extensively in \cite{DDdP1,DDdP2,DDdP3} by Dodds, Dodds and de Pagter in developing the theory of non-commutative Banach function spaces. 

Next we observe that \cite[Lemma 2.5]{FK} also holds for the weighted non-commutative decreasing rearrangements. We will also show the relationship between the weighted and ``unweighted'' non-commutative decreasing rearrangements, i.e. the singular value function. This relationship will be instrumental in our further results.
\\

In the final section we will take inspiration from \cite{DDdP1} and use the weighted non-commutative decreasing rearrangements to construct Banach function spaces. The main result of the article is that these spaces coincide with the weighted non-commutative Banach function spaces of Labuschagne and Majewski in \cite{LM}. A convex function $\psi:[0,\infty)\mapsto[0,\infty]$ which satisfies $\psi(0)=0$ and $\lim_{u\rightarrow\infty}\psi(u)=\infty$, is neither identically zero or infinitely valued on all of $(0,\infty)$ and is left continuous at $b_\psi = \{u>0:\psi(u)<\infty\}$, is referred to as an \emph{Orlicz function}. If $(X,\Sigma,m)$ is a $\sigma$-finite measure space, then the Orlicz space associated with $\psi$ is defined as the set
$$L^\psi=\{f\in L^0(X,\Sigma,m):\psi(\lambda\vert f\vert)\in L^1(X,\Sigma,m) \text{ for some } \lambda>0\}.$$ 
Under the Luxemburg-Nakano norm, given by
$$\Vert f\Vert_\psi = \inf\{\lambda>0:\Vert\psi\big(\tfrac{\vert f\vert}{\lambda}\big)\Vert_1\leq 1\},$$
an Orlicz space $L^\psi$ is a Banach function space. 

In statistical physics, entropy may not be well defined for every state $f\in L^1(X,\Sigma,m)$, especially in large systems. However, entropy is well defined for every $f\in L\log(L+1)$ and as such is an appropriate choice for the space containing the states of the system (see \cite{LM2} for details). The set of all regular random variables was shown to coincide with the subspace of $L^{\cosh-1}$ of zero expectation random variables \cite{PS}. The space $L^{\cosh-1}$ turns out to be an equivalent renorming of the K\"othe dual of $L\log(L+1)$. (See \cite{LM2})

In \cite{LM} Labuschagne and Majewski proved a non-commutative analogue of the Pistone-Sempi theorem, i.e. that the set of regular observables with respect to a state $x\in L^1(\Mtilde)$ agrees with the Orlicz space $L^{\cosh-1}_x(\Mtilde)$ (see \cite[Definition 3.5, Theorem 3.8]{LM}). 

In this article $\M$ will be a semi-finite von Neumann algebra equipped with a semi-finite normal faithful trace $\tau$. The lattice of projections will be denoted $\M_P$.

Given a self-adjoint operator $a\in\Mtilde$, the notation $e_B(a)$ will be used to denote the spectral projection of $a$ corresponding to the Borel set $B$. 

We will denote the Lebesgue measure by $m$ and the measure given by $\mu_t(x)dt$ by $\nu$.

\section{The map $\tau_x$}
Let $\M$ be a von Neumann algebra equipped with a semi-finite normal faithful trace $\tau$ and $0\neq x\in L^1(\Mtilde,\tau)+\M$. Recall that the map $\tau_x$ is defined by 
$$
\tau_x:\Mtilde\mapsto \mathbb{R}:a\mapsto \int_0^\infty\mu_t(a)\mu_t(x)dt.
$$

In \cite{LM} the authors assumed $x\in L^1(\Mtilde,\tau)$, but we will relax this assumption to that stated above. In the light of \cite[Proposition 2.6]{DDdP2}, the results obtained in \cite{LM}, in particular \cite[Theorem 3.7]{LM}, is still valid in this broader context. We will elaborate on the modifications needed in the proof of \cite[Theorem 3.7]{LM} for our broader context when they become necessary.

The map $\tau_x$ resembles a semi-finite, normal faithful trace. More specifically we have the following result from \cite[Proposition 3.10]{LM}.

\begin{prop} For any non-zero $0\leq x\in L^1(\Mtilde,\tau)+\M$, the map $\tau_x:\Mtilde\mapsto \mathbb{R}:a\mapsto \int_0^\infty\mu_t(a)\mu_t(x)dt$ has the following properties for all $0\leq a\in\Mtilde$:
\begin{enumerate}
\item $\tau_x$ is subadditive, homogeneous and satisfies $\tau_x(a^*a)=\tau_x(aa^*)$;
\item $\tau_x(a)=0$ implies $a=0$;
\item $\sup_n\tau_x(a_n)=\tau_x(a)$ for every sequence $\{a_n\}$ increasing to $a$ in $\Mtilde^+$.
\item if $x\in L^1(\Mtilde,\tau)$, then $\tau_x(1)<\infty$.
\end{enumerate}
\end{prop}

Many of the basic properties one can prove for a trace can also be proven for $\tau_x$, often by using the same proof, albeit with minor modifications where needed. The following two properties will be useful, and so we provide their proofs.

\begin{lemma}\label{lemequiv} If $p,q\in \M_P$ and $p\sim^v q$, then $\tau_x(p) = \tau_x(q)$.
\end{lemma}

\begin{proof} $\tau_x(p) = \tau_x(v^*v) = \tau_x(vv^*) = \tau_x(q)$.
\end{proof}

The previous lemma is identical to the result in the tracial case. The proof is provided largely as a short demonstration of some of the technical similarities that $\tau_x$ shares with that of a trace. The next lemma, however, is not entirely identical. The reader may recall that the inequality that follows is in fact an equality when dealing with a trace. 

\begin{lemma}\label{lemorth} For $p,q\in\M$, if $p\wedge q=0$, then $\tau_x(p)\leq\tau_x(1-q)$.
\end{lemma}

\begin{proof} If $p\wedge q=0$, then 
$p=1-p^\perp=(p\wedge q)^\perp - p^\perp = p^\perp \vee q^\perp -p^\perp \sim q^\perp -q^\perp\wedge p^\perp\leq q^\perp = 1-q.$
Therefore by Lemma \ref{lemequiv} $\tau_x(p)\leq \tau_x(1-q)$.
\end{proof}

Recall the definition of a $\tau$-measurable operator.

\begin{definition}\emph{ A closed operator $a$ affiliated with $\M$ is $\tau$-measurable if and only if for all $\delta>0$ there exists a projection $p\in \M$ such that $p\Hbt\subset D(a)$, $\Vert ap\Vert<\infty$ and $\tau(1-p)\leq \delta$.}
\end{definition}

By substituting $\tau_x$ into the role of the trace, we can define $\tau_x$-measurable operators.

\begin{definition}\emph{ A closed operator $a$ affiliated with $\M$ is $\tau_x$-measurable if and only if for all $\delta>0$ there exists a projection $p\in \M$ such that $p\Hbt\subset D(a)$, $\Vert ap\Vert<\infty$ and $\tau_x(1-p)\leq \delta$. We denote the set of all $\tau_x$ measurable operators $\Mtilde_x$.}
\end{definition}

In the proof of the next theorem we will rely on the fact that $\mu_t(x)$ is non-negative for all $t>0$.

\begin{theorem}\label{MisMx} An operator $a$ is $\tau$-measurable if and only if $a$ is $\tau_x$-measurable.
\end{theorem}	

\begin{proof} Suppose $a\in \Mtilde$ and let $\delta>0$. Since $x\in L^1(\Mtilde)+\M$ we have that $\int_0^t\mu_s(x)ds<\infty$ for every $t>0$ \cite[Proposition 2.6]{DDdP2}. As a function of $t$, the quantity $\int_0^t\mu_s(x)ds$ is continuous, increasing and $\int_0^0\mu_s(x)ds=0$. Therefore we can find an $\epsilon>0$ such that $\int_0^\epsilon\mu_s(x)ds\leq \delta$.

 Since $a\in\Mtilde$ there exists $p\in\M_P$ such that $p\Hbt \subset D(a)$, $\Vert ap\Vert<\infty$ and \\$\tau(1-p)\leq \epsilon$. Then 
\begin{eqnarray*}
\tau_x(1-p)	& = & \int_0^{\tau(1-p)}\mu_s(x)ds
\\			& \leq & \int_0^\epsilon\mu_s(x)ds
\\			& \leq & \delta.
\end{eqnarray*}

It follows that $a$ is $\tau_x$ measurable and hence $\Mtilde \subset \Mtilde_x$.

 Suppose that $a\in \Mtilde_x$ and let $\delta>0$. For each $n\in \mathbb{N}$, let $\epsilon_n=\frac{1}{2^n}$ and then let $p_n$ be a projection such that $p_n\Hbt\subset D(a)$, $\Vert ap_n\Vert<\infty$ and $\tau_x(1-p_n)\leq \frac{1}{2^n}$. Clearly $\tau_x(1-p_n)\downarrow 0$ as $n\rightarrow \infty$.

If there exists an $n$ such that $\tau(1-p_n) = 0$, then by the faithfulness of $\tau$, $p_n = 1$ and $a\in \M\subset\Mtilde$. So we may assume $\tau(1-p_n)>0$ for all $n\in \mathbb{N}$.

As is pointed out in the proof of \cite[Theorem 3.7]{LM}, the measure $\nu=\mu_t(x)dt$ is mutually absolutely continuous to the Lebesgue measure on the interval $[0,t_x)$, when $t_x = \inf\{t>0:\mu_t(x)=0\}<\infty$, or on $[0,\infty)$ when $\mu_t(x)>0$ for all $t>0$. Since we assumed that $\tau(1-p_n)>0$ for each $n\in\mathbb{N}$, it follows that $\int_0^{\tau(1-p_n)}\mu_t(x)dt>0$ for all $n\in\mathbb{N}$.

 Now suppose $\tau(1-p_n)\geq k > 0$ for some $k>0$ and all $n\in\mathbb{N}$. Then
$$\tau_x(1-p_n) = \int_0^{\tau(1-p_n)}\mu_s(x)ds \geq \int_0^k\mu_s(x)ds = \beta >0$$ 
for all $n >0$. Thus we have a contradiction. It follows that $\inf\{\tau(1-p_n):n\in\mathbb{N}\}=0$. And thus there exists an $n$ such that $p_n\Hbt\subset D(a)$, $\Vert ap_n\Vert<\infty$ and \\$\tau(1-p_n)\leq \delta$. Therefore we have that $\Mtilde = \Mtilde_x$.
\end{proof}

It is therefore not necessary to refer to $\Mtilde_x$.

\begin{remark} Now $\tau_x$ is a positive, unitarily invariant functional on $\M$, but is not a trace, or even a weight in the sense of von Neumann algebras.

Let $\M = L^\infty([0,2],dt)$ be equipped with the usual trace and $x=f(t)=\exp(-t)$. Note that $\mu_t(x) = \exp(-t)$.

Then $\tau_x(\chi_{[0,2)}) = 1-e^{-2} < 2(1-e^{-1})=\tau_x(\chi_{[0,1)})+\tau_x(\chi_{[1,2)})$, and therefore $\tau_x$ does not induce a measure on $[0,\infty)$ in the same way that $\tau$ does. $\tau_x$ is nevertheless close enough to being a trace to ensure that we may construct very well behaved spaces using this functional.
\end{remark}

\section{Weighted non-commutative decreasing rearrangements}

To define the weighted non-commutative decreasing rearrangement of an operator $a\in\Mtilde$, we use the approach of Fack and Kosaki in \cite{FK}. Many of the proofs in this section are similar to the corresponding ones found in \cite{FK}. In the cases where our new context does not change the proof of a statement to any appreciable degree, we will omit the proof.

\begin{definition}\label{wdr}\emph{For $a\in \Mtilde$, we define the function $\mu(a,x):[0,\infty)\mapsto [0,\infty]:t\mapsto\mu_t(a,x)$ by 
$$
\mu_t(a,x)=\inf\{\Vert ae\Vert: e\in\M_P, \tau_x(1-e)\leq t\}.
$$
We also define $d(a,x):[0,\infty)\mapsto [0,\infty]:t\mapsto d_t(a,x)$ by 
$$
d_t(a,x)=\tau_x(e_{(t,\infty)}(\vert a\vert)).
$$}
\end{definition}

\begin{lemma} $d(a,x)$ is decreasing.
\end{lemma}

\begin{proof} Let $t_1 \leq t_2$. Then $e_{(t_1,\infty)}(\vert T\vert) \geq e_{(t_2,\infty)}(\vert T\vert)$.
 
 By the monotonicity of $\tau_x$, it follows that 
 $$d_{t_1}(T,x) = \tau_x(e_{(t_1,\infty)}(\vert T\vert)) \geq \tau_x(e_{(t_2,\infty)}(\vert T\vert)) = d_{t_2}(T,x).$$
\end{proof}

\begin{lemma} $d(a,x)$ is right continuous.
\end{lemma}
\begin{proof} Suppose $t_i \downarrow t$. Then $e_{t_i}(\vert a\vert)\downarrow_{SO} e_t(\vert a\vert)$ in the strong operator topology. So $e_{(t_i,\infty)}(\vert a\vert) \uparrow e_{(t,\infty)}(\vert a\vert)$. Since $\tau_x$ is normal, it follows that \\
$d_{t_i}(S,x) = \tau_x(e_{(t_i,\infty)}(\vert a\vert)) \uparrow \tau_x(e_{(t,\infty)}(\vert a\vert)) = d_{t}(a,x)$.  
\end{proof}

Proposition \ref{muinf} and Lemmas \ref{muinf2}  and \ref{list} are proved in \cite{FK} for the tracial case. In our context the proofs remain virtually unchanged and so we omit them.

\begin{prop}\label{muinf} For $a\in\Mtilde$, we have that $\mu_t(a,x) = \inf\{s\geq 0 : d_s(a,x)\leq t\}$. Moreover $\mu_t(a,x)$ is non-increasing and right continuous. Also $d_{\mu_t(a,x)}(a,x)\leq t$ for all $t\geq0$.
\end{prop}

\begin{lemma}\label{muinf2} For each $t\geq 0$, let $R_t(x)$ be the set of all $\tau$-measurable operators $b$ such that $\tau_x(\sup(\vert b\vert))\leq t$. 
For $a\in \Mtilde$, we have that 
$$\mu_t(a,x) = \inf\{\Vert a - b\Vert: b\in R_t(x)\}.$$
\end{lemma}

\begin{lemma}\label{list} For $a\in\Mtilde$ and $x\in L^1(\Mtilde)$,

\begin{enumerate}

\item $\lim_{t\downarrow 0}\mu_t(a,x) = \Vert x \Vert$.
\item $\mu_t(a,x) = \mu_t(\vert a\vert, x) = \mu_t(a^*,x)$ and $\mu_t(\lambda a,x) = \vert\lambda\vert\mu_t(a,x)$.
\item $\mu_{t+s}(a+b,x) \leq \mu_t(a,x) + \mu_s(b,x)$.
\item $\mu_{t+s}(ab,x) \leq \mu_t(a,x)\mu_s(b,x)$.

\end{enumerate}
\end{lemma}

The following result gives us the relationship between $\mu(a,x)$ and the non-commutative decreasing rearrangement $\mu(a)$. This will play an important role in achieving the further results presented in this article.

\begin{theorem}\label{rearangements} Let $a\in\Mtilde$ and consider $\mu_t(a)\in L^0([0,\infty),\nu)$, where $\nu$ is the measure given by $\mu_t(x)dt$. Then the decreasing rearrangement of $\mu(a)$ with respect to $\nu$ is $\mu(a,x)$.
\end{theorem}

\begin{proof} We denote the distribution function of a function $f\in L^0([0,\infty),\nu)$ with respect to $\nu$ by $d(f,\nu)$ and the decreasing rearrangement with respect to $\nu$ by $\mu(f,\nu)$. We will calculate $\mu(T,x)$ and $\mu_t(\mu(T),\nu)$ using the prescriptions in Lemma \ref{muinf} and \cite[Definition 2.1.5]{BS} respectively.

It is well known that $d_t(\mu(\vert a\vert)) = d_t(\vert a\vert)$. Since $\mu(a)$ is decreasing and therefore $\chi_{(t,\infty)}(\mu(a))=\chi_{[0,d_t(\mu(\vert a\vert)))}$, it follows that 
\begin{align*}
d_t(\mu(a),\nu)	& =  \nu(\chi_{(t,\infty)}(\mu(a)))
\\				& =  \int_0^\infty\chi_{(t,\infty)}(\mu(a))(s)\mu_s(x)ds
\\				& =  \int_0^\infty\chi_{[0,d_t(\mu(\vert a\vert)))}(s)\mu_s(x)ds
\\				& =  \int_0^{d_t(\mu(\vert a\vert))}\mu_s(x)ds
\\				& =  \int_0^{d_t(\vert a\vert)}\mu_s(x)ds
\\				& =  \int_0^\infty \chi_{[0,\tau(e_{(t,\infty)}(\vert a\vert))}(s)\mu_s(x)ds
\\				& =  d_t(a,x).
\end{align*}
By using Proposition \ref{muinf} to calculate $\mu_t(a,x)$ and \cite[Definition 2.1.5]{BS} to calculate $\mu(\mu(a),\nu)$, it is clear that $\mu_t(a,x)=\mu(\mu(a),\nu)$.
\end{proof}

From Theorem \ref{rearangements} we immediately have the corollary.

\begin{cor}\label{L11} For $a\in\Mtilde$, 
$$\int_0^\infty\mu_t(T,x)dt = \int_0^\infty \mu_t(a)\mu_t(x)dt,$$
i.e. 
$$\int_0^\infty\mu_t(a,x)dt = \tau_x(a)$$
\end{cor}

\begin{proof} Since $x\in L^1(\Mtilde) + \M$, we have that $\int_0^n\mu_t(x)dt<\infty$ for all $n\in\mathbb{N}$. It follows that $\nu$ is a $\sigma$-finite measure. The corollary then follows from Theorem \ref{rearangements} and \cite[Chapter 2, Proposition 1.8]{BS} with $p=1$.
\end{proof}

Corollary \ref{L11} can also be proved following the conventional approach as was done in \cite[Proposition 2.7]{FK}, albeit with modifications where necessary. 

\section{Weighted non-commutative Banach function spaces}\label{banach fn}

For the sake of the reader we recall the definition of a weighted non-commutative Banach function space.

\begin{definition}\label{WBFS}\emph{\cite[Definition 3.6]{LM} Let $x\in L^1(\Mtilde,\tau)+\M$, and let $\rho$ be a rearrangement-invariant Banach function norm on $L^0((0,\infty),\mu_t(x)dt)$.  Then the weighted non-commutative Banach function space is defined as $L_x^\rho(\Mtilde,\tau) = \{a\in\Mtilde:\mu(a)\in L^\rho((0,\infty),\mu_t(x)dt)\}$.}

\end{definition}

As was mentioned previously, in \cite[Theorem 3.7]{LM} it was assumed that $x\in L^1(\Mtilde)$. In our broader context, where $x\in L^1(\Mtilde)+\M$, some minor modifications are necessary to the proof given in \cite{LM}. First we need to point out that the function 
$$F_x(t) = \int_0^t\mu_s(x)ds$$
is continuous and strictly increasing on $[0,t_x)$, where $t_x=\inf\{s>0:\mu_s(x)=0\}$, and constant on $[t_x,\infty)$ (when $t_x<\infty$). The modification is that in our context we have that $F_x$ is a homeomorphism from $[0,t_x)$ to $[0,\tau(x))$, where it is now possible that $\tau(x)=\infty$.

We also need the measure $\nu$ to be non-atomic. As pointed out in \cite{LM}, $\nu$ is mutually absolutely continuous to the Lebesgue measure $m$. To show that $\nu$ is non-atomic let $E$ be a Borel set. We now need to consider the two possibilities $\nu(E)<\infty$ and $\nu(E)=\infty$. If $\nu(E)<\infty$, we argue exactly as in \cite{LM} to find a measurable subset $F$ of $E$ with $0<\nu(F)<\nu(E)$. If $\nu(E)=\infty$, we choose an interval $[0,a]$ such that for $F=E\cap [0,a]$, $\nu(F)>0$. Then $0<\nu(F) = \int_0^a \chi_E\mu_t(x)dt<\infty$ and therefore $0<\nu(F)<\nu(E)$.

In light of this, we can now state \cite[Theorem 3.7]{LM} in our context.

\begin{theorem}\label{WBFS2}\cite[Theorem 3.7]{LM} Let $x\in L^1(\Mtilde,\tau)+\M$, and let $\rho$ be a rearrangement-invariant Banach function norm on $L^0((0,\infty),\mu_t(x)dt)$ that satisfies the following conditions for any measurable subset $E$ of $(0,\infty)$ and $f,f_n\in L^0((0,\infty),\mu_t(x)dt)^+$:
\begin{enumerate}
\item if $0\leq f_n\uparrow f$ $\nu$-a.e. then $\rho(f_n)\uparrow\rho(f)$
\item if $\int_E\mu_t(x)dt<\infty$, then $\rho(\chi_E)<\infty$
\item if  $\int_E\mu_t(x)dt<\infty$, then  $\int_Ef(t)\mu_t(x)dt\leq C_E\rho(f)$.
\end{enumerate} for some $C_E>0$ dependent on $E$ and $\rho$, but independent of $f$. Then $L^\rho_x(\Mtilde)$ is a linear space and $\Vert \cdot\Vert_\rho:a\mapsto \rho(a)$ is a norm on $L^\rho_x(\Mtilde)$. Equipped with the norm $\Vert \cdot\Vert_\rho$ the space $L^\rho_x(\Mtilde)$ is a Banach space that injects continuously into $\M$.
\end{theorem}

We propose the following alternative definition for a weighted space. Notice that if $\tau_x$ is a trace, Definition \ref{weightedspace} is the standard definition of a non-commutative Banach function space.

\begin{definition}\label{weightedspace}\emph{ Let $x\in L^1(\Mtilde)+\M$ and $\rho$ a Banach function function norm on $L^0([0,\infty))$ that satisfies the following conditions for any measurable subset $E$ of $(0,\infty)$ and $f,f_n\in L^0([0,\infty))^+$,
\begin{enumerate}
\item if $0\leq f_n\uparrow f$ $m$-a.e. then $\rho(f_n)\uparrow\rho(f)$
\item if $\int_Edt<\infty$ then $\rho(\chi_E)<\infty$
\item if $\int_Edt<\infty$ then $\int_Ef(t)dt\leq C_E\rho(f)$.
\end{enumerate} for some $C_E>0$ dependent on $E$ and $\rho$, but independent of $f$. The space $L^\rho(\Mtilde,\tau_x)$ is the space of all $a\in\Mtilde$ such that $\mu(a,x)\in L^\rho([0,\infty))$.}
\end{definition}

We are not yet claiming that $L^\rho(\Mtilde,\tau_x)$ is a Banach space. It is possible, after some preliminary justifications, to show this fact using \cite[Theorem 5.1.19]{BS}, however, we will be following a different route. We also want to show the relationship our spaces have with those defined in \cite{LM}. As it happens, in doing so, we also discover, as a corollary, that the spaces $L^\rho(\Mtilde,\tau_x)$ are Banach spaces.

\begin{theorem}\label{equiv} Let $L^\rho_x(\Mtilde)$ be a weighted non-commutative Banach function space. Then there exists a rearrangement invariant Banach function norm $\bar{\rho}$ in the sense of \cite{BS} on $L^0([0,\infty))$ such that $L^\rho_x(\Mtilde) = L^{\bar{\rho}}(\Mtilde,\tau_x)$. 

Conversely, for the space $L^\rho(\Mtilde,\tau_x)$, there exists a weighted non-commutative Banach function space $L^{\bar{\rho}}_x(\Mtilde)$ such that $L^\rho(\Mtilde,\tau_x) = L^{\bar{\rho}}_x(\Mtilde)$.
\end{theorem}

\begin{proof} Let $\rho$ be a Banach function norm over $L^0([0,\infty),\nu)$. We have that $\rho$ is rearrangement invariant and that $\nu$ is nonatomic from the proof of \cite[Theorem 3.7]{LM}, and therefore $L^0([0,\infty),\nu)$ is a resonant measure space. It then follows from the Luxemburg representation theorem \cite[Theorem 2.4.10]{BS} that there exists a rearrangement invariant Banach function norm $\bar{\rho}$ over $L^0([0,\infty))$ such that for all $f\in L^0([0,\infty),\nu)$
$$\rho(f) = \bar{\rho}(\mu(f,\nu)),$$

where $\mu(f,\nu)$ is the decreasing rearrangement of $f$ with respect to the measure $\nu$.

Since for all $a\in\Mtilde$ it was shown that $\mu(a,x)$ is the decreasing rearrangement of $\mu(a)$ with respect to $\nu$, it follows that
$$\rho(\mu(a)) = \bar{\rho}(\mu(a,x)).$$

Therefore $a\in L^\rho_x(\Mtilde)$ if and only if $a\in L^{\bar{\rho}}(\Mtilde,\tau_x)$, and we also have that $\Vert a\Vert_\rho = \Vert a\Vert_{\bar{\rho}}$. 

Given a Banach function norm $\rho$ on $L^0([0,\infty))$ (in the sense of \cite{BS}), the second part follows from setting $\bar{\rho}(f) = \rho(\mu(f,\nu))$ for all $f\in L^0([0,\infty),\nu)^+$. We must show that $\bar{\rho}$ is a rearrangement invariant Banach function norm on $L^0([0,\infty),\nu)$ in the sense of \cite{BS}. Now for $0\leq f\in L^0([0,\infty),\nu)$ we have that $f=0$ $\nu$-a.e. if and only if $\mu(f,\nu)=0$ $m$-a.e. if and only if $\bar\rho(f)=\rho(\mu(f,\nu))=0$. 

Let $0\leq g\leq f$ $\nu$-a.e. in $L^0([0,\infty),\nu)$. Then $\mu(g,\nu)\leq \mu(f,\nu)$ $m$-a.e. and therefore $\bar{\rho}(g)=\rho(\mu(g,\nu))\leq \rho(\mu(f,\nu))=\bar{\rho}(f)$. 

Suppose we have that $0\leq f_n\uparrow f$ $\nu$-a.e., then $\mu(f_n,\nu)\uparrow\mu(f,\nu)$ by \cite[Proposition 2.1.7]{BS} and therefore $\bar{\rho}(f_n)=\rho(\mu(f_n,\nu))\uparrow \rho(\mu(f,\nu)) = \bar{\rho}(f)$. 

Now suppose we have a Borel set $E\subset [0,\infty)$ with $\nu(E)<\infty$. We want to show that $\bar{\rho}(\chi_E)<\infty$. But $\bar{\rho}(\chi_E)=\rho(\mu(\chi_E))=\rho(\chi_{[0,\nu(E))})<\infty$ since $\nu(E)<\infty$ and therefore $\int_{[0,\nu(E))}dt<\infty$. Furthermore we have that there exists $C_E>0$, dependent on $[0,\nu(E))$ and $\rho$, such that $\int_0^{\nu(E)}\mu(f,\nu)dt \leq C_E\rho(\mu(f,\nu))$ for all $0\leq f\in L^0([0,\infty),\nu)$. Then it follows from \cite[Theorem 2.2.2]{BS}, with $g=\chi_E$, that $\int_E fd\nu\leq \int_0^{\nu(E)}\mu_t(f,\nu)dt\leq C_E\bar{\rho}(f)$.

We also need to show that $\bar{\rho}$ is subadditive. Let $f,g\in L^0([0,\infty),\nu)$. Then if we consider the commutative von Neumann algebra $L^\infty([0,\infty),\nu)$, we can conclude from \cite[Theorem 3.4]{DDdP1} that $\vert\mu(f+g,\nu)-\mu(g,\nu)\vert \prec\prec \mu(f,\nu)$. Then it is routine to show that $\rho(\mu(f+g,\nu))\leq \rho(\mu(f,\nu)) + \rho(\mu(g,\nu))$, i.e. $\bar{\rho}(f+g)\leq \bar{\rho}(f) + \bar{\rho}(g)$ (see \cite[Theorem 2.4.6]{BS}). So we have that $\bar{\rho}$ is a Banach function norm satisfying the conditions in Theorem \ref{WBFS2}.

Let $f,g\in L^0([0,\infty),\nu)$ such that $\mu(f,\nu) = \mu(g,\nu)$. Then since $\bar{\rho}(f) = \rho(\mu(f,\nu))=\rho(\mu(g,\nu))=\bar{\rho}(g)$ we have that $\bar{\rho}$ is rearrangement invariant.  

For all $a\in\Mtilde$ it follows from Theorem \ref{rearangements} that $\bar{\rho}(a) = \bar{\rho}(\mu(a))=\rho(\mu(a,x))$, and therefore $L^\rho(\Mtilde,\tau_x) = L^{\bar{\rho}}_x(\Mtilde)$.
\end{proof}

It should be noted that, strictly speaking, the equality in the above theorem is set equality. However, that $L^\rho(\Mtilde,\tau_x)$ is a Banach space follows as a corollary. 

\begin{cor} Let $\rho$ be a Banach function norm satisfying the conditions in Definition \ref{weightedspace}. The space $L^\rho(\Mtilde,\tau_x)$ is a Banach space with respect to the norm $a\mapsto\rho(a)$ for all $a\in L^\rho(\Mtilde,\tau_x)$ that injects continuously into $\Mtilde$.
\end{cor}

\begin{proof} To see that $a\mapsto\rho(a)$ is a norm, let $\bar{\rho}$ be the Banach function norm for which $L^\rho(\Mtilde,\tau_x)=L^{\bar{\rho}}_x(\Mtilde)$ setwise. 
Then for all $a\in L^\rho(\Mtilde,\tau_x)$, we have that $\rho(a) = \bar{\rho}(a)$. Since $\bar{\rho}$ acts as a norm on $L^{\bar{\rho}}_x(\Mtilde)$, it follows that $\rho$ acts as a norm on $L^\rho(\Mtilde,\tau_x)$. Furthermore we know that $L^{\bar{\rho}}_x(\Mtilde)$ is a Banach space with respect to the norm $a\mapsto\bar{\rho}(a)$ that injects continuously into $\Mtilde$, and hence the corollary follows.  
\end{proof}

\section{Weighted non-commutative Orlicz spaces}

Recall that an Orlicz function is a convex function $\psi:[0,\infty)\mapsto [0,\infty]$ satisfying $\psi(0)$ and $\lim_{t\mapsto\infty}\psi(t)=\infty$ that is neither identically zero, nor infinitely valued on all of $(0,\infty)$. In addition $\psi$ must be left continuous at $b_\psi=\sup\{t>0:\psi(t)<\infty\}$. Given an Orlicz function $\psi$ and a measure space $(X,\Sigma,m)$, the Orlicz space $L^\psi(X,\Sigma,m)$ is defined as 
$$L^\psi(X,\Sigma,m)=\{f\in L^0(X,\Sigma,m):\psi(\lambda\vert f\vert)\in L^1(X,\Sigma,m) \text{ for some } \lambda>0\}.$$ 
For a given Orlicz space $L^\psi(X,\Sigma,m)$, the Luxemburg-Nakano norm is given by 
$$\Vert f\Vert_\psi = \inf\{\lambda>0:\Vert\psi\big(\tfrac{\vert f\vert}{\lambda}\big)\Vert_1\leq \}.$$
$L^\psi(X,\Sigma,m)$ equipped with the Luxemburg-Nakano is a Banach function space. 

Given an Orlicz function $\psi$, we can construct the following three types of non-commutative Orlicz spaces:
\begin{enumerate}
\item $L^\psi(\M)=\{a\in\M:\mu(a)\in L^\psi([0,\infty))$;
\item\label{two} $L^\psi_x(\M)=\{a\in\M:\mu(a)\in L^\psi([0,\infty),\mu_t(x)dt)$;
\item\label{three} $L^\psi(\M,\tau_x)=\{a\in\M:\mu(a,x)\in L^\psi([0,\infty))$.
\end{enumerate}

As a demonstration of the importance of the weighted Orlicz spaces $L^\psi_x(\Mtilde)$,  it was shown in \cite{LM} that the set of regular observables with respect to a state $x\in L^1(\Mtilde)$ agrees with the Orlicz space $L^{\cosh-1}_x(\Mtilde)$ (see \cite[Definition 3.5, Theorem 3.8]{LM}).

Given either a space $L_x^\rho(\Mtilde)$ or a space $L^\rho(\Mtilde,\tau_x)$, in Theorem \ref{equiv} we only showed that there exists an equivalent function norm. In the case of weighted Orlicz spaces, we can go further and precisely describe the equivalent function norm that would generate the same space.

For this we will need the following lemma (Compare \cite[Lemma 2.1]{LM}).

\begin{lemma} Let $\psi$ be an Orlicz function and $a\in \Mtilde$ such that $\psi(\vert a\vert)\in\Mtilde$. Set $\psi(\infty)=\infty$ to extend $\psi$ to a function on $[0,\infty]$. Then $\psi(\mu_t(a,x)) = \mu_t(\psi(\vert a\vert),x)$.
\end{lemma}

\begin{proof} As was noted by \cite[Remark 2.3.1]{FK} in their context, the proof of Lemma \ref{muinf} easily adapts to show that when computing weighted decreasing rearrangements according to the prescription given in Definition \ref{wdr}, we may restrict our attention to the commutative von Neumann algebra generated by the spectral projections of $\vert a\vert$. 

From the Borel functional calculus and the proof of \cite[Lemma 2.1]{LM}, it is known that $\psi(\Vert \vert a\vert e\Vert)=\Vert \psi(\vert a\vert)e\Vert$, where values of infinity are allowed. Therefore
\begin{align*}
\mu_t(\psi(a),x)	& =  \inf\{\Vert \psi(\vert a\vert)e\Vert:e\in\M_P, \tau_x(1-e)\leq t\}
\\					& =  \inf\{\psi(\Vert \vert a\vert e\Vert):e\in\M_P, \tau_x(1-e)\leq t\}
\\					& =  \psi(\inf\{\Vert \vert a\vert e\Vert:e\in\M_P, \tau_x(1-e)\leq t\})
\\					& =  \psi(\mu_t(a,x)). 
\end{align*}
\end{proof}

\begin{theorem}\label{Orlicz} Let $\psi$ be an Orlicz function. Then $L_x^\psi(\Mtilde) = L^\psi(\Mtilde,\tau_x)$.
\end{theorem}
\begin{proof} Recall that
$$L^\psi_x(\Mtilde) = \{a\in\Mtilde: \mu(a)\in L^\psi([0,\infty),\nu)\}.$$
Therefore $a\in L^\psi_x(\Mtilde) $ if and only if $\mu(a)\in L^\psi([0,\infty),\nu)$ if and only if there exist an $\lambda>0$ such that $\psi(\lambda\mu(a))\in L^1([0,\infty),\nu)$. So we have that
\begin{align*}
a\in L^\psi_x(\Mtilde)	& \Leftrightarrow  \text{ there exists }\lambda>0 \text{ such that } \int_0^\infty \psi(\lambda\mu_t(a))\mu_t(x)dt <\infty
\\						& \Leftrightarrow 	\text{ there exists } \lambda>0 \text{ such that } \int_0^\infty \mu_t(\psi(\lambda a))\mu_t(x)dt <\infty
\\						& \Leftrightarrow 	\text{ there exists } \lambda>0 \text{ such that } \int_0^\infty \mu_t(\psi(\lambda a),x)dt <\infty
\\						& \Leftrightarrow 	\text{ there exists } \lambda>0 \text{ such that } \int_0^\infty \psi(\lambda\mu_t(a,x))dt <\infty
\\						& \Leftrightarrow  a\in L^\psi(\Mtilde,\tau_x).
\end{align*}

Similarly for the norm of the weighted Orlicz space we have that
\begin{align*}
\Vert a \Vert_{x,\psi} 	& = \inf\{\lambda>0: \Vert\psi(\mu(a)/\lambda)\Vert_{1,\nu} \leq 1\}
\\						& = \inf\{\lambda>0: \int_0^\infty\psi(\mu(a)/\lambda)\mu_t(x)dt \leq 1\}
\\						& =  \inf\{\lambda>0: \int_0^\infty(\psi(\mu_t(a,x)/\lambda))dt \leq 1\}
\\						& =  \Vert a\Vert_{\tau_x,\psi}.
\end{align*}
\end{proof}

From Theorem \ref{Orlicz} it follows that $L^p_x(\Mtilde)=L^p(\Mtilde,\tau_x)$ for $1\leq p \leq\infty$ and that $L_x^{\cosh-1}(\Mtilde) = L_x^{\cosh-1}(\Mtilde,\tau_x)$. Recall that $L_x^{\cosh-1}(\Mtilde)$ coincides with the set of regular observables, and therefore so does $L^{\cosh-1}(\Mtilde,\tau_x)$. The equivalence of the weighted non-commutative Banach function spaces and the relationship between the weighted and non-weighted non-commutative decreasing rearrangements presented in this article provides a powerful new functional calculus with which to refine the above mentioned application.

\section*{Acknowledgements}

The author thanks Louis E. Labuschagne for guidance given and many useful discussions and also for proofreading this article.

\noindent Funding: This work was supported by the North-West University.

\end{document}